\documentclass[reqno]{amsart}
\usepackage{amsmath,amssymb,cite}
\usepackage[mathscr]{euscript}
\usepackage{mathtools}
\usepackage[scr=boondox,  
            cal=esstix]   
           {mathalpha}
\usepackage{xcolor}
\usepackage{hyperref}
\usepackage{comment}
\usepackage{dsfont}
\usepackage{tikz-cd}
\usepackage[normalem]{ulem}
\usepackage{bbold,dsfont}

\definecolor{bblue}{rgb}{.2,0.2,.8}

\theoremstyle{plain}
\newtheorem{theorem}{Theorem}[section]
\newtheorem{proposition}[theorem]{Proposition}
\newtheorem{lemma}[theorem]{Lemma}
\newtheorem{corollary}[theorem]{Corollary}

\theoremstyle{definition}
\newtheorem{definition}[theorem]{Definition}
\newtheorem{assumption}[theorem]{Assumption}

\theoremstyle{remark}
\newtheorem{remark}[theorem]{Remark}

\numberwithin{equation}{section}
\numberwithin{theorem}{section}

\def\be{\begin{equation}}
\def\ee{\end{equation}}
\def\bp{\begin{pmatrix}}
\def\ep{\end{pmatrix}}
\def\bea{\begin{eqnarray}}
\def\eea{\end{eqnarray}}

\def\\{\par\medskip}

\let\0=\noindent

\newcommand*{\defeq}{\mathrel{\vcenter{\baselineskip0.5ex \lineskiplimit0pt
                     \hbox{\scriptsize.}\hbox{\scriptsize.}}}%
                     =}

\renewcommand{\epsilon}{\varepsilon}

\title[A classification of $\mathbb{F}_{p_k}$-braces]{A classification of  $\mathbb{F}_{p^k}$-braces using bilinear forms}

\author{Riccardo Aragona}
\address{\noindent Riccardo Aragona \hfill\break\indent 
 DISIM, Universit\`a dell'Aquila
\hfill\break\indent 
67100 Coppito, L'Aquila, Italy
}
\email{riccardo.aragona@univaq.it}

\author{Giuseppe Nozzi}
\address{\noindent Giuseppe Nozzi \hfill\break\indent 
	DISIM, Universit\`a dell'Aquila
	\hfill\break\indent 
	67100 Coppito, L'Aquila, Italy
}
\email{giuseppe.nozzi@graduate.univaq.it}

\begin{document}

\thanks{All  the   authors  are   members  of   INdAM-GNSAGA
  (Italy).} 
\begin{abstract}
 Let $\mathbb{F}_{p^k}$ be a finite  field of odd characteristic $p$. In this paper we give a classification, up to isomorphism, of the commutative, $3$-nilpotent $\mathbb{F}_{p^k}$-algebras, starting from the connection with their bi-brace structure. 
 Such classification is the generalization in odd characteristic of the result proved by Civino at al.\ in characteristic $2$.
\end{abstract}

\noindent

\maketitle
\thispagestyle{empty}
\section{Introduction}
Let $\mathbb{F}_{p^k}$ be a finite  field of odd characteristic $p$ and let $V$ be an $n$-dimensional vector space over $\mathbb{F}_{p^k}$. Starting from elementary, abelian and regular subgroups of $\mathrm{AGL}(V,+)$ it is possible to construct an operation on $V$, denoted by $\circ$, such that the triple $(V,+,\circ)$ is a bi-brace of abelian type~\cite{RUMP2007153}. It is well known (see e.g. \cite{CDVS,MR3982254}) that there exists  a bijection between bi-braces, regular subgroups of $\mathrm{AGL}(V,+)$ normalized by $(V,+)$ and commutative, $3$-nilpotent  rings $(V,+,\cdot)$ , where the product is defined by
    $a\cdot b=a\circ b-a-b$.

After the definition of braces given by Rump in \cite{RUMP2007153} as a generalization of Jacobson radical rings, in recent years their study has interested many researchers due to the  connections of such structures to several other algebraic research topics. Some examples of such applications are the use of braces for analyzing nondegenerate involutive set-theoretical solutions of the Yang–Baxter equation \cite{BCJ,RUMP2007153}, or even for studying  the Hopf-Galois structures (see, e.g. ~\cite{hopfstef,hopfven}), and regular subgroups of the holomorph (see, e.g.~\cite{MR3465351,CDVS}). Another interesting application of braces is the one to cryptanalysis of block ciphers, under the hypothesis that $V\cdot V$ is a one-dimensional subspace of $V$, which was recently investigated in \cite{civino2,civino1,fedelecritto}. 

In according with the cryptanalytic application, our goal is to classify the isomorphism classes of $(V,+,\cdot)$, with $V\cdot V$ is a one-dimensional subspace. In particular we prove that the isomorphism classes of such algebras are completely determined by the equivalence classes of the symmetric bilinear forms $b:V\times V\longrightarrow \mathbb{F}_{p^k}$, defined by $b(a,b)=a\cdot b$. 
Such result generalizes in odd characteristic the classification proved in ~\cite{fedelecritto} for characteristic $2$.

The paper is organized as follows. In Section \ref{sec2}, we give the necessary background on brace theory, keeping particular emphasis to the relation between two-sided braces, radical rings and regular subgroups of the holomorph. In Section \ref{sec3}, starting from a vector space $V$ on $\mathbb{F}_{p^k}$, we define a skew bi-brace of abelian type, we give the properties of the associated commutative 3-nilpotent ring $(V,+,\cdot)$ and we describe the elements of the corresponding elementary abelian regular subgroup  of the holomorph of $V$. Finally, in Section \ref{sec4}, we show that the isomorphism classes of the commutative, $3$-nilpotent  $\mathbb{F}_{p^k}$-algebras $(V,+,\cdot)$ such that $V\cdot V$ has dimension $1$ are completely determined by the equivalence classes of the bilinear forms defined by the products of $V$.

\section{Preliminaries}\label{sec2}
In this section we give the basic notions on \emph{braces}, first defined by Rump~\cite{RUMP2007153}, and we recall some one-to-one correspondences between braces and other well-known algebraic structures.
\begin{definition}
    A skew (left) brace is a triple $(G,+,\circ)$, where $(G,+)$ and $(G,\circ)$ are groups and 
    \begin{equation}\label{leftskewbrace}
     a \circ (b + c)= a\circ b -a +a\circ c.
\end{equation}

A skew (right) brace is the analogous satisfying the condition 
\begin{equation}\label{rightskewbrace}
     (a+b)\circ c= a\circ c -c +b\circ c.
\end{equation}
\end{definition}

A skew (left/right) brace with abelian additive group $(G,+)$ is called \emph{skew brace of abelian type} (or simply \emph{(left/right) brace}).

A skew brace $(G,+,\circ)$ is said to be a \emph{skew two-sided brace} (or simply \emph{skew brace}) if  both Equation \eqref{leftskewbrace} and Equation \eqref{rightskewbrace} are satisfied, i.e.\ if it is a skew left and a skew right brace.
Notice that any skew brace such that its multiplicative group $(G,\circ)$ is an abelian group is two-sided.

Throughout this paper we will always consider braces with abelian multiplicative group. \medskip

Let $(G,+,\circ)$ be a brace,  it is possible to define a radical ring $(G,+,\cdot)$ by setting $a\cdot b\defeq a\circ b -a-b$, for $a,b\in G$ (see for example \cite{RUMP2007153}).\medskip

Let us denote by $\sigma:G\longmapsto \mathrm{Sym}(G)$ the right regular representation of $(G,+)$, with $\sigma_a:g\mapsto g+a$. Recall that the holomorph of a group $G$ is defined as the group
 \begin{equation*}
     \mathrm{Hol}(G)=\mathrm{Aut}(G)\ltimes \sigma(G)
 \end{equation*}
and it is isomorphic to the normaliser $N_{\mathrm{Sym}(G)}(\sigma(G))$ of $\sigma(G)$ in $\mathrm{Sym}(G)$.
 
Let $T_\circ$ be a regular  subgroup of $\mathrm{Hol}(G)$. We can give a labelling of the elements of $T_\circ$ by the elements of $G$ via the bijection $\tau$ defined by $a\mapsto \tau_a$, where $\tau_a$ is the unique element in $T_\circ$ sending $0$ to $a$. Since $T_\circ\leq\mathrm{Hol}(G)$, then each element $\tau_a\in T_\circ$ can be written uniquely as a product of an element $\gamma_a\in\mathrm{Aut}(G)$ and an element $\sigma_a\in\sigma(G)$.

We can define the regular subgroup $(G,\circ)$ of $\mathrm{Hol}(G)$ with the operation $\circ$ defined by
\begin{equation}
    a\circ b\defeq b\tau_a\quad \text{for each }a,b\in G,
\end{equation}
obtaining that $(G,\circ)$ and $T_\circ$ are isomorphic via the bijection $\tau$. In particular, we have a one to one correspondence between braces with additive group $(G,+)$ and regular  subgroups of $\mathrm{Hol}(G)$. For more details, see~\cite[Proposition 2.3]{MR3465351}.

\begin{definition}\label{bi-brace}
A brace $(G,+,\circ)$ is called \emph{bi-brace} if $(G,\circ,+)$ is also a brace.
\end{definition}

The same construction as above gives rise to a bijection between bi-braces $(G,+,\circ)$ and regular subgroups of $\mathrm{Hol}(G)$ normalized by $\sigma(G)$ (see~\cite[Theorem 3.1]{MR4130907}).
\begin{theorem}[\cite{CDVS,MR3982254}]\label{thmbrace}
    Let $(G,+)$ be an abelian group. The following data are equivalent:
    \begin{enumerate}
        \item[1.] a bi-brace $(G,+,\circ)$;
        \item[2.] an abelian, regular subgroup $N$ of $\mathrm{Hol}(G,+)$ normalized by $\sigma(G)$;
        \item[3.] a commutative, $3$-nilpotent  ring $(G,+,\cdot)$.
    \end{enumerate}
\end{theorem}\medskip

Finally we conclude this section with the following useful definitions
\begin{definition}\label{gammafunction}
    The function $\gamma:(G,\circ)\longmapsto \mathrm{Aut}(G)$ defined by $\gamma_a:b\mapsto -a+a\circ b$ is called \emph{gamma function} (for more details see e.g., \cite{MR3647970}).
\end{definition}
In \cite{MR3647970} it has been proved that regular subgroups of the holomorph are in bijective correspondence with these functions.
\begin{definition}[\cite{MR4648557}]
    A bi-brace $(G,+,\circ)$ is called $\mathbb{F}_{p^k}$-brace if $(G,+)$ is an $\mathbb{F}_{p^k}$-vector space and  the values of the gamma function are automorphisms of $\mathbb{F}_{p^k}$-vector spaces.
\end{definition}

It is worth to highlight that if $(G,+,\circ)$ is a two-sided brace, the notion of $\mathbb{F}_{p^k}$-braces coincides with the classical notion of $\mathbb{F}_{p^k}$-vector spaces (see Example 2 in~\cite{MR4648557}).

\section{The setting}\label{sec3}
Let $p$ be an odd prime integer and $n\ge 2$ be a positive integer. Let $V$ be an $n$-dimensional vector space over $\mathbb{F}_{p^k}$. Let denote us by $\mathbb{F}_{p^k}^\times$ the nonzero elements of $\mathbb{F}_{p^k}$. Let $\mathrm{Sym}(V)$ be the group of all the permutations on $(V,+)$ and $\mathrm{GL}(V,+)$ the group of invertible $\mathbb{F}_{p^k}$-endomorphisms on $(V,+)$. Let us denote by $T_+$ the translation group of   $(V,+)$, i.e.
\begin{equation*}
    T_+=\lbrace \sigma_a:x\mapsto x+a \mid a\in V \rbrace \leq \mathrm{Sym}(V)
\end{equation*}
and by $\mathrm{AGL}(V,+) \defeq \mathrm{GL}(V,+)\ltimes T_+=\mathrm{Hol(V)}$ the affine group of $V$. It is well known that $T_+$ is an elementary abelian regular subgroup of $\mathrm{Sym}(V)$ isomorphic to $(V,+)$ and $\mathrm{AGL}(V,+)$ is the normaliser of $T_+$ in $\mathrm{Sym}(V,+)$. 


\begin{theorem}[Dixon]\cite{MR0281782}\label{thm:dixon}
Let $X$ be a finite set and $\Phi,\Psi$ be regular subgroups of $Sym(X)$. If $\Phi\cong \Psi$, there exists $g\in Sym(X)$ such that $\Psi=\Phi^g=g^{-1}\Phi g$.
\end{theorem}

Let $T_\circ$ be  be elementary abelian regular subgroup of $\mathrm{AGL}(V,+)$. By Theorem \ref{thm:dixon}, there exists $g\in Sym(V)$ such that $T_\circ=T_+^g$. If $a\in V$, we can denote by $\tau_a$ the unique permutation in $T_\circ$ sending  $0$ to $a$, and so
\begin{equation*}
    T_\circ=\lbrace \tau_a \mid a\in V\rbrace.
\end{equation*}
It is possible to define an abelian group $(V,\circ)$ isomorphic to $T_\circ$, where the operation $\circ$ is defined by
\begin{equation}\label{defcirc}
    a\circ b=a\tau_b\quad \text{for any}\ a,b\in V.
\end{equation}
Since, by hypothesis, $T_\circ$ is a subgroup of $\mathrm{AGL}(V,+)$, then each element $\tau_a\in T_\circ$ can be written in a unique way as a product of an element of $GL(V,+)$ and an element of $T_+$, in other words
\begin{equation}\label{formatau}
    \tau_a=\gamma_a\sigma_a,
\end{equation}
where $\gamma_a\in \mathrm{GL}(V,+)$ and $\sigma_a\in T_+$. Notice that the homomorphism
\begin{equation*}
    \gamma:(V,\circ)\longmapsto \mathrm{GL}(V,+),\ a\mapsto (\gamma_a:b\mapsto -a+a\circ b)
\end{equation*}
is precisely the gamma function defined in Equation~\eqref{gammafunction}, and $(V,+,\circ)$ is an $\mathbb{F}_{p^k}$-brace.

Notice that for all $a,b,c\in V$
\begin{align*}
    (a+b)\circ c&=(a+b)\tau_c\\
    &=(a+b)\gamma_c\sigma_c\\
    &=a\gamma_c+b\gamma_c + c\\
    &=(a\circ c)+(b\circ c)-c
\end{align*}
and so, since $(V,+)$ and $(V,\circ)$ are abelian groups,  $(V,+,\circ)$ is a skew brace of abelian type. 

For any $a\in V$ we can define the endomorphism $\delta_a=\gamma_a-\mathbb{1}_V$ of $V$ and we can define a product on $V$ by
\begin{equation}\label{defprodotto}
    a\cdot b=a\delta_b\quad \text{for}\ b\in V.
\end{equation}
Notice that by Equations \eqref{formatau} and \eqref{defprodotto}, for every $a,b\in V$, we have
\begin{equation}\label{eq:circ}
    a\circ b=a+b+a\cdot b.
\end{equation}

\begin{definition}
    A ring $(V,+,\cdot)$ is said to be a \emph{(Jacobson) radical} ring if every element $a,b\in V$ is invertible with respect to the circle operation defined by $a\circ b=a+b+a\cdot b$.
\end{definition}
Finally the following result holds (see also \cite{CDVS}).

\begin{theorem}\label{thm:radical}
$(V,+,\cdot)$ is a commutative, associative $\mathbb{F}_{p^k}$-algebra such that the resulting ring is radical.
\end{theorem}
   In order to give a bi-brace structure over $(V,+,\circ)$, throughout the paper we will assume the following assumption.
\begin{assumption}\label{assum1}
    $T_+ <\mathrm{AGL}(V,\circ)\cong \mathrm{N}_{\mathrm{Sym(V)}}(T_\circ)$.
\end{assumption}
In \cite[Lemma 3]{CDVS}, the authors proved that $[\sigma_a,\tau_b]=\sigma_{a\cdot b}$ for every $a,b\in V$, from which  it follows that $T_+ <\mathrm{AGL}(V,\circ)$ if and only if $\sigma_{a\cdot b}\in T_\circ$. Hence, under the Assumption \ref{assum1}, by Equation \eqref{eq:circ} we have
    \begin{equation*}
        \sigma_{a\cdot b}\in T_\circ \cap T_+ =\lbrace \sigma_x\mid x\in \mathrm{Ann}(V)\rbrace,
    \end{equation*}
    where $\mathrm{Ann}(V)=\{a\in V \mid a\cdot V=0\}$ is the annihilator of the algebra $(V,+,\cdot)$. Hence $T_+ <\mathrm{AGL}(V,\circ)$ if and only if $a\cdot b\cdot c=0$ for all $a,b,c\in V$. Therefore 
    \begin{equation}\label{VVinAnn}
    V\cdot V\subseteq \mathrm{Ann}(V)
    \end{equation}
    and as a consequence we obtain $c\delta_{a\cdot b}=c\cdot a\cdot b=0$ for every $a,b,c\in V$, and so
    \begin{equation*}
        \gamma_{a\cdot b}=\mathbb{1}_n\quad \mathrm{for\ all\ } a,b\in V.
    \end{equation*}
By Equation \eqref{eq:circ}, Definition \ref{gammafunction} and reminding that $V^3=0$, we get that for all $x,a,b\in V$
\begin{align*}
 x\gamma_{a+b}&=-a-b+x\circ(a+b)\\
 &=-a-b+x+(a+b)+x\cdot(a+b) \\
 &=x+x\cdot a + x\cdot b
\end{align*}
and
\begin{align*}
 x\gamma_{a\circ b}&=-(a\circ b)+x\circ(a\circ b)\\
 &=-(a+b+a\cdot b)+x\circ (a+b+a\cdot b) \\
 &=-(a+b+a\cdot b) + x+(a+b+a\cdot b)+x\cdot (a+b+a\cdot b)\\
 &=x+x\cdot a + x\cdot b
\end{align*}
Thus, $T_+ <\mathrm{AGL}(V,\circ)$ if and only if 
\begin{equation*}
   \gamma_{a+b}=\gamma_{a\circ b}=\gamma_a\gamma_b
\end{equation*}
from which follows that $(V,+,\circ)$ is a bi-brace.\ In other words we proved the following result, which generalizes \cite[Proposition 4.1]{MR3982254}.
\begin{proposition}\label{propbibrace}
    Let $(V,+,\cdot)$ be a nilpotent $\mathbb{F}_{p^k}$-algebra. The triple $(V,+,\circ)$ is a bi-brace if and only if $V^3=0$.
\end{proposition}
We conclude this section constructing a matrix form for the endomorphism $\gamma_a$, $a\in V$. Let $\{e_1,\ldots,e_n\}$ be the canonical basis of $V$ and let us set $d$ the dimension of $\mathrm{Ann}(V)$.\ Without lost of generality we can assume that $\mathrm{Ann}(V)=\mathrm{span}\lbrace e_{m+1},\dots,e_n\rbrace$ with $m=n-d$. For the case $k=1$ of the following result, see also \cite[Theorem 3.11]{calderini}.
\begin{theorem}\label{formagamma}
For every $a\in V$ there exists a matrix $\Theta_a\in (\mathbb{F}_{p^k})^{m\times d}$ such that 
\begin{equation*}
    \gamma_a=\begin{pmatrix}
        \mathbb{1}_m & \Theta_a\\
        0& \mathbb{1}_d
    \end{pmatrix}.
\end{equation*}
Moreover, $\Theta_a$ is the zero matrix for every $a\in \mathrm{Ann}(V)$.
\end{theorem}
\begin{proof}
Notice that $a\in \mathrm{Ann}(V)$ if and only if $a\circ b=a+b$ for all $b\in V$. It follows that $e_j\circ e_i=e_j+e_i$ for $j=m+1,\dots,n$ and $i=1,\dots, n$. In other words
\begin{equation}\label{1}
    e_j\circ e_i=e_j\gamma_{e_i}+e_i=e_j+e_i,
\end{equation}
on the other hand for $j=1,\dots,m$ and $i=1,\dots, m$
\begin{equation}\label{2}
   e_j\circ e_i=e_j\gamma_{e_i}+e_i=e_j+e_i+e_j\cdot e_i
\end{equation}
where $e_j\cdot e_i\in \mathrm{Ann}(V)=\mathrm{span}\{e_{m+1},\dots,e_n\}$. Thus, by Equation \eqref{1} and \eqref{2} we obtain 
\begin{equation*}
    \gamma_{e_i}=\begin{pmatrix}
        \mathbb{1}_m& \Theta_{e_i}\\
        0&\mathbb{1}_d
    \end{pmatrix}\ \text{for }i=1,\dots,m
\end{equation*}
and $\gamma_{e_i}=\mathbb{1}_n$ for $i=m+1,\dots,n$. Moreover, denoting by $\Theta_{i,j}$ the $j$-$th$ row of the matrix $\Theta_i$, we get that $e_i\cdot e_j=(\underbrace{0,\dots,0}_m,\Theta_{i,j})$. Finally, for every $a=a_1e_1+\dots+a_ne_n\in V$ 
\begin{equation*}
\gamma_a=\begin{pmatrix}
    \mathbb{1}_m& a_1\Theta_{e_1}+\dots+a_m\Theta_{e_m}\\
    0&\mathbb{1}_d
\end{pmatrix}.\qedhere
\end{equation*}
\end{proof}

\section{Isomorphism classes of \texorpdfstring{$(V,+,\cdot)$}{(V,+,)}}\label{sec4}
In this section $(V,+,\cdot)$ will denote an $n$-dimensional commutative, $3$-nilpotent $\mathbb{F}_{p^k}$-algebra with product defined in Eq. \eqref{defprodotto} and  $\mathrm{Ann}(V)=\mathrm{span}\{e_{m+1},\dots,e_n\}$ of dimension $d=n-m$. Let $\Theta_{i}$ be the $m\times d$ matrix with entries in $\mathbb{F}_{p^k}$ associated to the vector $e_i$ of the canonical basis of $V$, defined  in Theorem \ref{formagamma}. Recall that $\Theta_{j}=0$ for $j=m+1,\dots,n$.
\begin{definition}
We shall say that the $m\times m$ matrix $\Theta\defeq [\Theta_{1} \dots \Theta_{m}]$ with entries in $(\mathbb{F}_{p^k})^{d}$ is the \emph{defining matrix} with respect to the canonical basis $\{e_1,\dots,e_n\}$ for the algebra $(V,+,\cdot)$.
\end{definition}

\begin{lemma}\label{theorem:theoremtheta}
The defining matrix $\Theta$ of the algebra $(V,+,\cdot)$ has the following properties:
\begin{enumerate}
    \item $\Theta$ is symmetric;
    \item each $\mathbb{F}_{p^k}$-linear combination of  $\Theta_1,\dots,\Theta_m$ is non-zero, except the one with all zero coefficients.
\end{enumerate}
\end{lemma}
\begin{proof} 
By Theorem \ref{formagamma}, we have $e_i\cdot e_j=(\underbrace{0,\dots,0}_m,\Theta_{i,j})$ and, since $(V,+,\cdot)$ is a commutative algebra, $\Theta_{i,j}=\Theta_{j,i}$ and so $\Theta$ is symmetric.

Let us suppose that a non trivial $\mathbb{F}_{p^k}$-linear combination of $\Theta_1,\dots,\Theta_m$ is non-zero, in other words let
        \begin{equation*}
           \displaystyle\sum_{i=1}^{m} a_i\Theta_i =0\quad \ \text{with } a_i\in \mathbb{F}_{p^k}\ \text{and } a_s\neq 0,\text{ for some }s\in \lbrace 1,\dots,m\rbrace.
       \end{equation*}
        Then 
        \begin{equation*}
            \gamma_{a_1 e_1+\dots+a_m e_m}=
            \begin{pmatrix}
                \mathds{1}_m & a_1 \Theta_1+\dots+a_m\Theta_m\\
                0 & \mathds{1}_d
            \end{pmatrix}
            =\mathds{1}_n,
        \end{equation*}
and $a_1 e_1+\dots+a_m e_m\in \mathrm{Ann}(V)$, which is a contradiction since $\mathrm{Ann}(V)=\mathrm{span}\{e_{m+1},\dots,e_n\}$.
\end{proof}
Conversely we have the following.
\begin{theorem}\label{thm:theoremtheta}
Any $m\times m$ symmetric matrix $\Theta=[\Theta_1, \dots, \Theta_m]$, with entries in $(\mathbb{F}_{p^k})^d$, such that each non trivial $\mathbb{F}_{p^k}$-linear combination of $\Theta_1, \dots, \Theta_m$ is non-zero, is the defining matrix with respect to the canonical basis of a commutative, $3$-nilpotent  $\mathbb{F}_{p^k}$-algebra, with $\mathrm{Ann}(V)=\mathrm{span}\{e_{m+1},\dots,e_n\}$ of dimension $d=n-m$.
\end{theorem}
\begin{proof}
Let $\Theta$ be an $m\times m$ symmetric matrix with entries in $(\mathbb{F}_{p^k})^d$ such that each $\mathbb{F}_{p^k}$-linear combination of its columns $\Theta_1,\dots,\Theta_m$ is non-zero, except the one with all zero coefficients.

For $a=a_1 e_1+\cdots +a_n e_n\in V$, let us  define the map $\tau_a=\gamma_a\sigma_a$, where
\begin{equation*}
    \gamma_a=\begin{pmatrix}
        \mathbb{1}_{m}&a_1\Theta_1+\cdots +a_{m}\Theta_{m}\\
        0 &\mathbb{1}_d
    \end{pmatrix},
\end{equation*}
and $\sigma_a$ is the translation of $(V,+)$ sending $0\mapsto a$.
Let $$T_\circ=\lbrace \tau_a\mid  a\in V\rbrace < \mathrm{AGL}(V,+),$$ and $\circ$ be the operation induced by $T_\circ$ on $V$ as in Equation \eqref{defcirc}.

By Theorem \ref{thmbrace} and Proposition \ref{propbibrace}, in order to prove that $(V,+,\cdot)$ is a commutative, $3$-nilpotent  algebra, it is enough to prove that $T_\circ$ is an abelian regular subgroup of $\mathrm{AGL}(V,+)$ normalized by $T_{+}=\{\sigma_a\mid a\in V\}$.
\begin{itemize}
    \item $T_\circ$ is a group. The neutral element is $\tau_0=\mathbb{1}_V$ and $\circ$ is associative by definition.
    We claim that $\tau_x^{-1}=\tau_x^{p-1}$. Indeed,  for every $a,b\in V$ we have
    \begin{align*}
    a\underbrace{\tau_b\dots\tau_b}_{\textit{p-times}}&=a\underbrace{\gamma_b\dots \gamma_b}_p+b\underbrace{\gamma_b\dots \gamma_b}_{p-1}+\dots+b\gamma_b+b\\
    &=a+b(\underbrace{\gamma_b\dots \gamma_b}_{p-1}+\underbrace{\gamma_b\dots \gamma_b}_{p-2}+\dots+\gamma_b+\mathds{1}_n),
\end{align*}
since $\gamma_b^p=\gamma_{pb}=\mathbb{1}_n$. Notice that $1+\dots +(p-1)=\displaystyle\frac{(p-1)p}{2}=kp$, for $k=\frac{p-1}{2}$ positive integer. Thus
\begin{align*}
    \underbrace{\gamma_b\dots \gamma_b}_{p-1}+\underbrace{\gamma_b\dots \gamma_b}_{p-2}+\dots+\gamma_b&=
    \begin{pmatrix}
        (p-1)\mathds{1}_m & \displaystyle\sum_{i=1}^{p-1} \Theta_b\\
        0 & (p-1)\mathds{1}_d
    \end{pmatrix}\\
    &=\begin{pmatrix}
        (p-1)\mathds{1}_m & 0\\
        0 & (p-1)\mathds{1}_d
    \end{pmatrix}
    \end{align*}
    and so $\underbrace{\gamma_b\dots \gamma_b}_{p-1}+\underbrace{\gamma_b\dots \gamma_b}_{p-2}+\dots+\gamma_b+\mathds{1}_n=0$.
It follows that $a\underbrace{\tau_b\dots\tau_b}_{p}=a$.
\item $T_\circ$ is abelian. Indeed, for every $a,b,c\in V$
\begin{align*}
a\tau_c\tau_b=a\gamma_c\sigma_c\gamma_b\sigma_b&=a\gamma_c\gamma_b+c\gamma_b+b\\
 &=a\gamma_b\gamma_c+b\gamma_c+c\\
 &=a\tau_b\tau_c.
\end{align*}
\item $T_\circ$ is regular, or in other words, for every $ a,b\in V$ there exists a unique $c\in V$ such that $a\tau_c=b$. We claim that $c=a\underbrace{\tau_a\dots\tau_a}_{p-2}\tau_b$. Indeed,
\begin{align*}
a\tau_c&=a\tau_{a\tau_a\dots\tau_a\tau_b}=a\gamma_{a\gamma_a\sigma_a\dots \gamma_a\sigma_a \gamma_b\sigma_b}\sigma_{a\tau_a\dots\tau_a\tau_b}\\
&=a\gamma_{a\underbrace{\gamma_a\dots \gamma_a}_{p-2} \gamma_b+a\underbrace{\gamma_a\dots \gamma_a}_{p-3}\gamma_b+\dots+a\gamma_b+b}+\\
    &+a\underbrace{\gamma_a\dots \gamma_a}_{p-2} \gamma_b+a\underbrace{ \gamma_a\dots \gamma_a}_{p-3} \gamma_b+\dots+a \gamma_a \gamma_b+a \gamma_b+b.
\end{align*}

Notice that for every $a,b\in V$ we have
$$\gamma_{a \gamma_b}=\gamma_{-b+b\circ a}=\gamma_{-b}\gamma_{b}\gamma_a= \gamma_a.$$ 
Therefore
\begin{align*}
    a\tau_c&=a\underbrace{ \gamma_a\dots \gamma_a}_{p-1} \gamma_b+a\underbrace{ \gamma_a\dots \gamma_a}_{p-2} \gamma_b+\dots+a \gamma_a \gamma_b+a \gamma_b+b\\
    &=a(\underbrace{ \gamma_a\dots \gamma_a}_{p-1}+\underbrace{ \gamma_a\dots \gamma_a}_{p-2}+ \gamma_a+\mathds{1}_n) \gamma_b+b\\
    &=b.
\end{align*}
Let us assume by contradiction that $c$ is not unique. Let $\Bar{c}\in V$, with $\Bar{c}\neq c $, such that $a\tau_{\Bar{c}}=b$ or analogously that $\Bar{c}\tau_a=b$.
\begin{equation*}
    \Bar{c}=b\tau_a^{-1}=b\underbrace{\tau_a\dots\tau_a}_{p-1}=a\tau_b\underbrace{\tau_a\dots\tau_a}_{p-2}=a\underbrace{\tau_a\dots\tau_a}_{p-2}\tau_b=c.
\end{equation*}
\item $T_+ <N_{\textit{Sym}(V)}(T_{\circ})$, or in other words, for every $a,b\in V$, $\sigma_a \tau_b\sigma_a^{-1}\in T_\circ$. Indeed
\begin{align*}
x\sigma_a\tau_b\underbrace{\sigma_a\dots\sigma_a}_{p-1}&=(x+a) \gamma_b\sigma_b\sigma_a\dots\sigma_a\\
    &=x \gamma_b+a \gamma_b+b+\underbrace{a+\dots+a}_{p-1}\\
    &=x\tau_{a \gamma_b+b+(p-1)a}, 
\end{align*}
since $\gamma_{a\gamma_b+b+(p-1)a}=\gamma_a\gamma_b\gamma_{(p-1)a}=\gamma_b$.\qedhere
\end{itemize}
\end{proof}
In the rest of the present section we consider true the following assumption, that allow us to study the product in \(V\) as a standard symmetric bilinear form. 
\begin{assumption}
    $V\cdot V$ is a one-dimensional subspace of $V$.
\end{assumption}
Let us define the symmetric bilinear form 
  \begin{equation*}
        \begin{array}{rccc}
b:&V\times V &\longrightarrow &\mathbb{F}_{p^k}\\
   &(a,b) &\longmapsto & a\cdot b
\end{array}
\end{equation*}

The following result directly follows from Lemma \ref{theorem:theoremtheta}.
\begin{corollary}
If $V\cdot V$ is a one-dimensional subspace of $V$,  then $\Theta$ is an $m\times m$ invertible matrix with coefficients in $\mathbb{F}_{p^k}$.
\end{corollary}

\begin{remark}\label{Ann=MM}
Notice that $\mathrm{Rad}(b)=\lbrace a\in V\mid b(a,V)=0\rbrace$  coincides with $\mathrm{Ann}(V)=\{a\in V\mid a\cdot V=0\}$. So, by Equation \eqref{VVinAnn}, we can decompose the annihilator $\mathrm{Ann}(V)=(V\cdot V)\oplus H$, for a suitable $(d-1)$-dimensional subspace $H$ of $V$. Thus, up to consider the quotient algebra $(V/H,+,\cdot)$, we can assume that $V\cdot V$ coincides with $\mathrm{Ann}(V)$ and spanned by the last vector of the canonical basis. We obtain the following decomposition of $V$
\begin{equation*}
    V=W\oplus (V\cdot V)
\end{equation*}
where $V\cdot V=\mathrm{span}\lbrace e_n \rbrace$.
\end{remark}
We notice that $\Theta=[\Theta_1,\dots,\Theta_{n-1}]$ is the matrix associated to the bilinear form $b$ restricted to the subspace $W$ of $V$ defined above.\medskip

Before proving the main theorem of this work, we recall some useful general definitions and some well known results concerning the classification of bilinear forms over finite fields of odd characteristic.

\begin{definition}
Let $b$ a nondegenerate symmetric bilinear form on an $\mathbb{F}_{p^k}$-vector space with associated matrix $B$. The discriminant of $b$ is defined to be the coset $\mathrm{det}(B)\mathbb{F}_{p^k}^{\times 2}$ in $\mathbb{F}_{p^k}^\times /\mathbb{F}_{p^k}^{\times 2}.$
\end{definition}

\begin{theorem}{\cite[Proposition 5, pg. 34]{serre}}\label{classificationbilforms}
Let $b$ be a nondegenerate symmetric bilinear form on a $\mathbb{F}_{p^k}$-vector space $V$ of dimension strictly bigger than 1, then there exists a basis for $V$ such that the matrix $B$ associated to $b$ has one of the following non-equivalent diagonal form 
\begin{equation*}
\left[\begin{array}{cccc}
1&0&0&0\\
0&1&0&\vdots\\
\vdots &0 &\ddots &0\\
0&\cdots&0&1\\
\end{array}\right]\text{ or  }
\left[\begin{array}{cccc}
1&0&0&0\\
0&\ddots&0&\vdots\\
\vdots &0 &1 &0\\
0&\cdots&0&q\\
\end{array}\right],
\end{equation*} 
where $q$ is a non-square element of $\mathbb{F}_{p^k}$.
\end{theorem}

\begin{corollary}{\cite[Corollary of Proposition 5, pg. 35]{serre}}\label{corbilform}
For two nondegenerate bilinear forms over $\mathbb{F}_{p^k}$ to be equivalent it is necessary and sufficient that they have the same rank and same discriminant.
\end{corollary}

We are ready to prove the one to one correspondence between the isomorphism classes of $(V,+,\cdot)$ and congruence classes of the corresponding defining matrices.

\begin{theorem}\label{isomorfismi}
Let $(V,+,\cdot_1)$ and $(V,+,\cdot_2)$ be two commutative, $3$-nilpotent  $\mathbb{F}_{p^k}$-algebras with $V\cdot_1 V=V\cdot_2 V=\mathrm{span}\lbrace e_n\rbrace$. A matrix with block form
\begin{equation*}
    \begin{pmatrix}
   A & 0\\
   0 & l
    \end{pmatrix},\quad A\in GL(m,p^k),\ l\in \{1,q\}
\end{equation*}
is an isomorphism between the algebras $(V,+,\cdot_1)$ and $(V,+,\cdot_2)$ if and only if
\begin{equation}\label{eqeccoqui}
    A\Theta^1A^{tr}=l\Theta^2,
\end{equation}
where $\Theta^1$ and $\Theta^2$ are the defining matrices of the algebras $(V,+,\cdot_1)$ and $(V,+,\cdot_2)$ respectively and $q$ is a non-square element of $\mathbb{F}_{p^k}^\times$.
\end{theorem}
\begin{proof}
Let $A$ be an $(n-1)\times (n-1)$ invertible matrix and $l\in \mathbb{F}_{p^k}^\times$ such that $A\Theta^1 A^{tr}=l\Theta^2$. For $i,j=1,\ldots,n-1$
\begin{align*}
    e_i\begin{pmatrix}
        A&0\\
        0&l
    \end{pmatrix}\cdot_1 e_j\begin{pmatrix}
        A&0\\
        0&l
    \end{pmatrix}&=(A_i,0)\cdot_1(A_j,0)\\
    &=(\underbrace{0,\dots,0}_{n-1},A_i\Theta_{i,j}^1A_j^{tr})\\
    &=(0,\dots,0,l\Theta_{i,j}^2)\\
    &=(0,\dots,0,\Theta_{i,j}^2)\begin{pmatrix}
        A&0\\
        0&l
    \end{pmatrix}\\
    &=e_i\cdot_2 e_j \begin{pmatrix}
        A&0\\
        0&l
    \end{pmatrix}.
\end{align*}
Viceversa if $\begin{pmatrix}
        A&0\\
        0&l
    \end{pmatrix}$ is an isomorphism between $(V,+,\cdot_1)$ and $(V,+,\cdot_2)$ then
    \begin{equation*}
        e_i\begin{pmatrix}
        A&0\\
        0&l
    \end{pmatrix}\cdot_1 e_j\begin{pmatrix}
        A&0\\
        0&l
    \end{pmatrix}=e_i\cdot_2e_j\begin{pmatrix}
        A&0\\
        0&l
    \end{pmatrix}.
    \end{equation*}
Lastly , if $(V,+,\cdot_1)$ and $(V,+,\cdot_2)$ satisfies Equation \eqref{eqeccoqui} for some invertible matrix $A$, then the bilinear forms associated to the products are equivalent.\ By Corollary~\ref{corbilform}, it is enough to consider $l\in \mathbb{F}_{p^k}^\times/\mathbb{F}_{p^k}^{\times 2}$ or equivalently $l\in \lbrace 1,q\rbrace$, where $q$ is a pre-selected non-square of  $\mathbb{F}_{p^k}$.\qedhere
\end{proof}

We are now ready to give a complete classification of the isomorphism classes of commutative, $3$-nilpotent  $\mathbb{F}_{p^k}$-algebras $(V,+,\cdot)$ with $V\cdot V$ of dimension 1.
\begin{theorem}\label{theorem:theoremmain1}
Let $(V,+,\cdot)$ be an $\mathbb{F}_{p^k}$-algebra as above. If\ $V\cdot V$ is an one-dimensional subspace of $V$, then there are two isomorphism classes of the algebra if $n-1$ is even and there is one class is $n-1$ is odd.
\end{theorem}
\begin{proof}
By Theorem \ref{classificationbilforms} it is enough to fix 
\begin{equation*}
 \Theta^1=
\left[\begin{array}{cccc}
1&0&0&0\\
0&1&0&\vdots\\
\vdots &0 &\ddots &0\\
0&\cdots&0&1\\
\end{array}\right],\ \Theta^2=
\left[\begin{array}{cccc}
1&0&0&0\\
0&\ddots&0&\vdots\\
\vdots &0 &1 &0\\
0&\cdots&0&q\\
\end{array}\right]
\end{equation*}
and by Proposition \ref{isomorfismi} it is enough to check whenever $\Theta^1$ and $q\Theta^2$ represent equivalent bilinear forms.
\begin{itemize}
    \item If $n-1$ is even, then  $\mathrm{det}(\Theta^1)=1$ is a square and $\mathrm{det}(q\Theta^2)=q^{n}$ is not a square. Then $\Theta^1$ and $q\Theta^2$ are not congruent.
    \item If $n-1$ is odd, then   $\mathrm{det}(\Theta^1)=1$ is a square and $\mathrm{det}(q\Theta^2)=q^{n}$ is a square. Then $\Theta^1$ and $q\Theta^2$ are congruent.\qedhere
\end{itemize}
\end{proof}

In conclusion, we notice that Theorem \ref{isomorfismi} and Theorem \ref{theorem:theoremmain1} only depend on the defining matrix of the algebras. Therefore the isomorphism classes of $(V,+,\cdot)$ are in one to one correspondence to the isomorphism classes of $V/H$, defined in Remark \ref{Ann=MM}. In other words the obtained results hold also in the case of an algebra $(V,+,\cdot)$ with $V\cdot V$ of dimension one and $\mathrm{Ann}(V)$ of any dimension.

\bibliographystyle{abbrv}
\bibliography{citation}
\end{document}